\documentclass[a4paper,11pt]{amsart}
\usepackage{amssymb}

\usepackage[colorlinks=true, linktocpage=false]{hyperref}
\numberwithin{equation}{section}
\newtheorem{theorem}{Theorem}[section]
\newtheorem{prop}[theorem]{Proposition}
\newtheorem{lem}[theorem]{Lemma}

\theoremstyle{remark}
\newtheorem{rem}[theorem]{Remark}
\newcommand{\R}{\mathbb{R}}

\newcommand{\N}{\mathbb{N}}

\author[M.~Ltifi]{Maroua Ltifi}
\address{Department of Mathematics, Faculty of Science of Gab\`es, university of Gab\`es; Tunisia}
\email{\sl widaltifi@gmail.com}

\title[Strong solution of modified 3D-Navier-Stokes equations ]
{Strong solution of modified 3D-Navier-Stokes equations}

\begin{document}
	\begin{abstract}
		In this paper we study the incompressible Navier-Stokes equations with a logarithm damping $\alpha \log(e+|u|^2)|u|^2u$ in $H^{1}$, where we used new methods, new tools and Fourier analysis.
	\end{abstract}

	

	\maketitle
	\tableofcontents

	
	\section{\bf Introduction}
	
	Recently, the  modified Navier-Stokes ¨Navier-Stokes equation with dumping¨ $\alpha|u|^{\beta-1}u$ was vastly studied by many  researchers:$$
	\begin{cases}
		\partial_t u
		-\Delta u+ u.\nabla u  +\alpha |u|^{\beta-1}u =\;\;-\nabla p\hbox{ in } \mathbb R^+\times \mathbb R^3\\
		{\rm div}\, u = 0 \hbox{ in } \mathbb R^+\times \mathbb R^3\\
	u(0,x)=u^{0}(x).
	\end{cases}
	$$
	It was firstly studied by Cai and Jiu \cite{XJ} in 2008. They are shown by Galerkin's methods the existence of global weak solution $u\in L^{\infty}(L^{2}(\R^3))\cap L^{2}(\dot{H}^{1}(\R^{3}))\cap L^{\beta+1}(L^{\beta+1}(\R^{3}))$ for $\beta\geq1$, global strong solution for any $\beta\geq\frac{7}{2}$ and that the strong solution is unique for any $\frac{7}{2}\leq\beta\leq5$.\\
	In this paper we study the large time behaviour in Fourier norms of the solution to the incompressible Navier-Stokes equations with  damping in three spatial dimensions in $H^{^1}$.
	\begin{theorem}\label{L1}
		Let $u^0\in H^1(\mathbb R^3)$ be a divergence free vector fields. For $\beta>3$  there is a unique global solution $u\in L^\infty(\R^+,H^1(\mathbb R^3)\cap C(\R^+,H^{-2}(\R^3))\cap L^2(\R^+,\dot H^2(\mathbb R^3))$ such that
		\begin{equation}\label{eqth01}\|u(t)\|_{L^2}^2+2\int_0^t\|\nabla u\|_{L^2}^2+2\alpha\int_0^t\|u\|_{L^{\beta+1}}^{\beta+1}\leq \|u^0\|_{L^2}^2.\end{equation}
		\begin{align}\label{eqth02}\nonumber
			\|\nabla u(t)\|_{L^2}^2+2\int_0^t\|\Delta u\|_{L^2}^2+\alpha(\beta-1) \int_{0}^{t}\||u|^{\beta-3} |\nabla|u|^2|^2\|_{L^{1}}\\+2\alpha\int_0^t\||u|^{\beta-1} |\nabla u|^2\|_{L^1}\leq\|\nabla u^0\|_{L^2}+\int_0^t\||u|^{2} |\nabla u|^2\|_{L^1}.\end{align}
	\end{theorem}
	Despite, the  continuity and uniqueness  of this modified equation is still a big open problem for $\beta=3.$ Indeed, the problem is restricted to the case $0<\alpha<\frac{1}{2}$ since the inequality
	\begin{align*}
		\frac{d}{dt}	\|\nabla u\|_{L^2}^{2}+\|\Delta u\|_{L^{2}}+2\alpha \int_{\R^3}|u|^2|\nabla u|^2\leq \int_{\R^3}|u|^2|\nabla u|^2
	\end{align*} is not resolvable for these values of $\alpha$ by classical methods and technics.\\
The trick that will be used in our statement is to increase the function $|u|^{2}u$ by a negligible function in relation to $|u|^{\xi}$, $\xi>0$.	
	\begin{equation}(NSD_{\log})\label{sys}
	\begin{cases}
		\partial_t u
		-\Delta u+ u.\nabla u  +\alpha \log(e+|u^2|)|u|^2u =\;\;-\nabla p\hbox{ in } \mathbb R^+\times \mathbb R^3\\
		{\rm div}\, u = 0 \hbox{ in } \mathbb R^+\times \mathbb R^3\\
		u(0,x) =u^0(x) \in H^{1}. \;\;\hbox{ in }\mathbb R^3,
	\end{cases}
	\end{equation}
	
	where $u=u(t,x)=(u_1,u_2,u_3)$ and $p=p(t,x)$ denote respectively the unknown velocity and the unknown pressure of the fluid at the point $(t,x)\in \mathbb R^+\times \mathbb R^3$, and $\alpha>0$. The terms $(u.\nabla u):=u_1\partial_1 u+u_2\partial_2 u+u_3\partial_3u$, while $u^0=(u_1^o(x),u_2^o(x),u_3^o(x))$ is an initial given velocity. If $u^0$ is quite regular, the divergence free condition determines the pressure $p$. We recall in our case it was assumed the viscosity is unitary ($\nu=1$) in order to simplify the calculations and the proofs of our results.\\
	
	The main result of our work is illustrated in the following theorem :
	
	\begin{theorem}\label{theo1} Let $u^0\in H^1(\mathbb R^3)$ be a divergence free vector fields, then there is a unique global  solution
		$u\in L^\infty(\R^+,H^1(\mathbb R^3)\cap C(\R^+,H^{-2}(\R^3))\cap L^2(\R^+,\dot H^2(\mathbb R^3))$ and $\alpha \log(e+|u|^2|u|^4, \frac{|u|}{e+|u|^2} |\nabla|u|^2|^2, \log(e+|u|^2) |\nabla|u|^2|^2, \log(e+|u|^2))|u|^2 |\nabla u|^2\in L^1(\R^+,L^1(\R^3))$. Moreover, for all $t\geq0$ \begin{equation}\label{eqth1}\|u(t)\|_{L^2}^2+2\int_0^t\|\nabla u\|_{L^2}^2+2\alpha\int_0^t\|\log(e+|u|^2)|u|^4\|_{L^1}\leq \|u^0\|_{L^2}^2.\end{equation}
		\begin{align}\label{eqth2}
			\nonumber\|\nabla u(t)\|_{L^2}^2+2\int_0^t\|\Delta u\|_{L^2}^2+\alpha \int_{0}^{t}\|\frac{|u|}{e+|u_n|^2} |\nabla|u|^2|^2\|_{L^{1}}+\alpha\int_0^t\| \log(e+|u|^2) |\nabla|u|^2|^2\|_{L^1} \\
			+2\alpha\int_0^t\|\log(e+|u|^2) |u|^2\nabla |u|^2\|_{L^1}\leq \|\nabla u^0\|_{L^2}^2e^{a_{\alpha}t}\end{align}
		where, $a_{\alpha}=e^{\frac{1}{2\alpha}}-e$
	\end{theorem}
\begin{rem}\label{r1}
If $u$ is the solution of (\ref{sys}), we have $u\in L^{4}(\R^{+},\dot{H}^{1}).$
	\end{rem}
The remainder of our paper is organized as follows. In the second section, we present some notations, definitions, and preliminary results. In Section 3, we will look at the global solution of Theorem \ref{theo1}. Furthermore, the solution's uniqueness and right continuity.
	\section{\bf Notations and preliminary results}
	\subsection{Notations} In this section, we collect some notations and definitions that will be used later.\\
	\begin{enumerate}
		\item[$\bullet$] The Fourier transformation is normalized as
		$$
		\mathcal{F}(f)(\xi)=\widehat{f}(\xi)=\int_{\mathbb R^3}\exp(-ix.\xi)f(x)dx,\,\,\,\xi=(\xi_1,\xi_2,\xi_3)\in\mathbb R^3.
		$$
		\item[$\bullet$] The inverse Fourier formula is
		$$
		\mathcal{F}^{-1}(g)(x)=(2\pi)^{-3}\int_{\mathbb R^3}\exp(i\xi.x)g(\xi)d\xi,\,\,\,x=(x_1,x_2,x_3)\in\mathbb R^3.
		$$
		\item[$\bullet$] The convolution product of a suitable pair of function $f$ and $g$ on $\mathbb R^3$ is given by
		$$
		(f\ast g)(x):=\int_{\mathbb R^3}f(y)g(x-y)dy.
		$$
		\item[$\bullet$] If $f=(f_1,f_2,f_3)$ and $g=(g_1,g_2,g_3)$ are two vector fields, we set
		$$
		f\otimes g:=(g_1f,g_2f,g_3f),
		$$
		and
		$$
		{\rm div}\,(f\otimes g):=({\rm div}\,(g_1f),{\rm div}\,(g_2f),{\rm div}\,(g_3f)).
		$$
		Moreover, if $\rm{div}\,g=0$ we obtain
		$$
		{\rm div}\,(f\otimes g):=g_1\partial_1f+g_2\partial_2f+g_3\partial_3f:=g.\nabla f.
		$$
		\item[$\bullet$] Let $(B,||.||)$, be a Banach space, $1\leq p \leq\infty$ and  $T>0$. We define $L^p_T(B)$ the space of all
		measurable functions $[0,t]\ni t\mapsto f(t) \in B$ such that $t\mapsto||f(t)||\in L^p([0,T])$.\\
		\item[$\bullet$] The Sobolev space $H^s(\mathbb R^3)=\{f\in \mathcal S'(\mathbb R^3);\;(1+|\xi|^2)^{s/2}\widehat{f}\in L^2(\mathbb R^3)\}$.\\
		\item[$\bullet$] The homogeneous Sobolev space $\dot H^s(\mathbb R^3)=\{f\in \mathcal S'(\mathbb R^3);\;\widehat{f}\in L^1_{loc}\;{\rm and}\;|\xi|^s\widehat{f}\in L^2(\mathbb R^3)\}$.\\
		\item[$\bullet$] For $R>0$, the Friedritch operator $J_R$ is defined by
		$$J_R(D)f=\mathcal F^{-1}({\bf 1}_{|\xi|<R}\widehat{f}).$$
		\item[$\bullet$] The Leray projector $\mathbb P:(L^2(\R^3))^3\rightarrow (L^2(\R^3))^3$ is defined by
		$$\mathcal F(\mathbb P f)=\widehat{f}(\xi)-(\widehat{f}(\xi).\frac{\xi}{|\xi|})\frac{\xi}{|\xi|}=M(\xi)\widehat{f}(\xi);\;M(\xi)=(\delta_{k,l}-\frac{\xi_k\xi_l}{|\xi|^2})_{1\leq k,l\leq 3}.$$
		\item[$\bullet$] $L^2_\sigma(\R^3)=\{f\in (L^2(\R^3))^3;\;{\rm div}\,f=0\}$.
		\item[$\bullet$] $\dot H^1_\sigma(\R^3)=\{f\in (\dot H^1(\R^3))^3;\;{\rm div}\,f=0\}$.
		\item[$\bullet$] $C_{r}(I,B)=\{f:I\rightarrow B\mbox{ right continuous }\}$ , where $B$ is Banach space and $I$ is an interval.
		\item[$\bullet$] Let $a\in\R,$ we define $a_{+}=\max(a,0)$.
	\end{enumerate}
	\subsection{Preliminary results}
	In this section, we recall some classical results and we give new technical lemmas.
	\begin{prop}(\cite{HBAF})\label{prop1} Let $H$ be Hilbert space.
		\begin{enumerate}
			\item If $(x_n)$ is a bounded sequence of elements in $H$, then there is a subsequence $(x_{\varphi(n)})$ such that
			$$(x_{\varphi(n)}|y)\rightarrow (x|y),\;\forall y\in H.$$
			\item If $x\in H$ and $(x_n)$ is a bounded sequence of elements in $H$ such that
			$$(x_n|y)\rightarrow (x|y),\;\forall y\in H.$$
			Then $\|x\|\leq\liminf_{n\rightarrow\infty}\|x_n\|.$
			\item If $x\in H$ and $(x_n)$ is a bounded sequence of elements in $H$ such that
			$$\begin{array}{l}
				(x_n|y)\rightarrow (x|y),\;\forall y\in H\\
				\limsup_{n\rightarrow\infty}\|x_n\|\leq \|x\|,\end{array}$$
			then $\lim_{n\rightarrow\infty}\|x_n-x\|=0.$
		\end{enumerate}
	\end{prop}
	\begin{lem}(\cite{JYC})\label{LP}
		Let $s_1,\ s_2$ be two real numbers and $d\in\N$.
		\begin{enumerate}
			\item If $s_1<d/2$\; and\; $s_1+s_2>0$, there exists a constant  $C_1=C_1(d,s_1,s_2)$, such that: if $f,g\in \dot{H}^{s_1}(\mathbb{R}^d)\cap \dot{H}^{s_2}(\mathbb{R}^d)$, then $f.g \in \dot{H}^{s_1+s_2-1}(\mathbb{R}^d)$ and
			$$\|fg\|_{\dot{H}^{s_1+s_2-\frac{d}{2}}}\leq C_1 (\|f\|_{\dot{H}^{s_1}}\|g\|_{\dot{H}^{s_2}}+\|f\|_{\dot{H}^{s_2}}\|g\|_{\dot{H}^{s_1}}).$$
			\item If $s_1,s_2<d/2$\; and\; $s_1+s_2>0$ there exists a constant $C_2=C_2(d,s_1,s_2)$ such that: if $f \in \dot{H}^{s_1}(\mathbb{R}^d)$\; and\; $g\in\dot{H}^{s_2}(\mathbb{R}^d)$, then  $f.g \in \dot{H}^{s_1+s_2-1}(\mathbb{R}^d)$ and
			$$\|fg\|_{\dot{H}^{s_1+s_2-\frac{d}{2}}}\leq C_2 \|f\|_{\dot{H}^{s_1}}\|g\|_{\dot{H}^{s_2}}.$$
		\end{enumerate}
	\end{lem}
	
	\begin{lem}\label{lem45}
		Let $A,T>0$ and $f,g,h:[0,T]\rightarrow\R^+$ three continuous functions such that
		\begin{align}\label{e1}\;\;\;\;\forall t\in[0,T];\;f(t)+\int_0^tg(z)dz & \leq A+\int_0^th(z)f(z)dz.\end{align}
		Then $$\forall t\in[0,T];\;f(t)+\int_0^tg(z)dz\leq A\exp(\int_0^th(z)dz).$$
	\end{lem}
	\begin{proof} By Gronwall lemma, we get
		$$\forall t\in[0,T];\;f(t)\leq A\exp(\int_0^th(z)dz).$$
		Put this inequality in \ref{e1} we obtain
		$$\begin{array}{lcl}
			\displaystyle f(t)+\int_0^tg(z)dz&\leq&\displaystyle  A+\int_0^th(z)A\exp(\int_0^zh(r)dr)dz\\
			&\leq&\displaystyle  A+A\int_0^th(z)\exp(\int_0^zh(r)dr)dz\\
			&\leq&\displaystyle  A+A\int_0^t\Big(\exp(\int_0^zh(r)dr)\Big)'dz\\
			&\leq&\displaystyle  A+A\Big(\exp(\int_0^th(r)dr)-1\Big)\\
			&\leq&\displaystyle  A\exp(\int_0^th(r)dr),
		\end{array}$$
		which ends the proof.
	\end{proof}
	\begin{lem}\label{lem46}
		Let $d\in\N$ Then, for all $x,y\in\R^d$, we have
		$$\langle\log(e+|x|^2)|x|^2x-\log(e+|y|^2)|y|^2y,x-y\rangle\geq0$$	
	\end{lem}
	\begin{proof}
		Let $a(z)=\log(e+|z|^2)|z|^2$ and $|x|\leq|y|$ :
		\begin{align*}
			\langle a(x)x-a(y)y,x-y\rangle=&\langle(a(x)-a(y))x,x-y\rangle+a(y)\langle x,x-y\rangle\\=&(a(x)-a(y))\langle x,x-y\rangle+a(y)|x-y|^{2}	
		\end{align*}	
		If $\langle x,x-y\rangle\geq0$ $(a(x)-a(y))\langle x,x-y\rangle+a(y)|x-y|^{2}\geq 0$\\Else,
		$\langle x,x-y\rangle<0$,	e.g:
		\begin{align*}
		(a(x)-a(y))\langle x,x-y\rangle+a(y)|x-y|^{2}	&=a(y)(|x-y|^{2}+\langle x,x-y\rangle) \\&=a(y)(\langle x-y,x-y\rangle-\langle x,x-y\rangle\\& a(y)\langle x-y,-y\rangle\\&=a(y)(|y|^2-\langle x,y\rangle)\\&\geq a(y)(|y|^2- |x||y|)\geq0 \end{align*}
		
	\end{proof}
	\begin{lem}\cite{JM}\label{lem5}
		Let $f:I\rightarrow\R$ be increasing function . Then there is $A\subset\R$ at most countable family such that for all $t$ in $A$,  $f$ is  discontinuous at $t$.
		Moreover, if $f$ is decreasing the $g=-f$.  	
	\end{lem}

	\section{\bf Existence and uniqueness of strong solution .}
	\subsection{Proof of Theorem $\ref{L1}$}
	To begin, we take the $L^{2}$ scalar product of the first equation with $u$ and integrate it on $[0,t] $, yielding \begin{align}\label{eq01}\|u(t)\|_{L^2}^2+2\int_0^t\|\nabla u\|_{L^2}^2+2\alpha\int_0^t \|u\|_{L^{\beta+1}}^{\beta+1} \leq \|u^0\|_{L^2}^2\end{align}
	Taking the $\dot{H}^{1}$ scalar product with $u$ as well :
		\begin{align*}
		\frac{1}{2}\frac{d}{dt}\|\nabla u\|_{L^{2}}^{2}+\|\Delta u\|^2_{L^2}+ \int_{\R^3} |u|^{\beta-3} |\nabla|u|^2|^2+\alpha \int_{\R^3} |u|^{\beta-1} |\nabla u|^2 &\leq \langle \nabla (u \nabla u),\nabla u\rangle_{L^{2}}\\	
		\frac{1}{2}\frac{d}{dt}\|\nabla u\|_{L^{2}}^{2}+\frac{1}{2}\|\Delta u\|^2_{L^2}+\frac{\alpha(\beta-1)}{2} \int_{\R^3} |u|^{\beta-3} |\nabla|u|^2|^2+\alpha \int_{\R^3} |u|^{\beta-1} |\nabla u|^2 &\leq\frac{1}{2}\int_{\R^3} |u|^{2} |\nabla u|^2
	\end{align*}
Integrate on $[0,t]$ we get
		\begin{align*}
		\|\nabla u(t)\|_{L^2}^2+2\int_0^t\|\Delta u\|_{L^2}^2+\alpha(\beta-1) \int_{0}^{t}\||u|^{\beta-3} |\nabla|u|^2|^2\|_{L^{1}}\\+2\alpha\int_0^t\||u|^{\beta-1} |\nabla u|^2\|_{L^1}\leq\|\nabla u^0\|_{L^2}+\int_0^t\||u|^{2} |\nabla u|^2\|_{L^1}.\end{align*}
$\bullet$For $\beta>3$ we obtain the global existence for bounded solution.\\
$\bullet$For $\beta=3$ Indeed, the problem is limited to the case $0<\alpha<\frac{1}{2}$ because the inequality (\ref{eqth02}) is unsolvable for these $\alpha$ values.
To solve our statement, we will add the function  $\log(e+|u|2)$ to $|u|^{2}u$.
We will solve the incompressible Navier-Stokes equations with logarithmic damping \ref{theo1} at the next party
\subsection{Proof of Theorem $\ref{theo1}$}
$$$$
	$\bullet$ {\bf{A priori estimates}}  \\
	We start by taking the $L^2$ scalar product of the first equation with $u$, we get \begin{align}\label{eq1}\|u(t)\|_{L^2}^2+2\int_0^t\|\nabla u\|_{L^2}^2+2\alpha\int_0^t\|\log(e+|u|^2)|u|^4\|_{L^1}\leq \|u^0\|_{L^2}^2\end{align}
	Also, taking the $\dot{H}^{1}$ scalar product of $(NSD_{\log})$ with $u$ :
	\begin{align*}
		\frac{1}{2}\frac{d}{dt}\|\nabla u\|^2+\|\Delta u\|^2_{L^2}+\frac{\alpha}{2} \int_{\R^3} \frac{|u|^2}{e+|u|^2} |\nabla|u|^2|^2+\frac{\alpha}{2} \int_{\R^3} \log(e+|u|^2) |\nabla|u|^2|^2&\\+ \int_{\R^3}(\frac{1}{2}-\alpha \log(e+|u|^2))|u|^2 |\nabla u|^2 &\leq 0
	\end{align*}
	For $t\geq0$, put $M_{t}=\{x\in\R^3:~~(\alpha\log(e+|u|^2)-\frac{1}{2})\geq0\}$.\\
	Further
	\begin{align*}
		\int_{\R^3}(\frac{1}{2}-\alpha \log(e+|u|^2))|u|^2 |\nabla u|^2&=\int_{M_{t}}(\frac{1}{2}-\alpha \log(e+|u|^2))|u|^2 |\nabla u|^2\\&+\int_{M^c_{t}}(\frac{1}{2}-\alpha \log(e+|u|^2))|u|^2 |\nabla u|^2.
	\end{align*}
	Since
	\begin{align*}
		\int_{M^c_{t}}(\frac{1}{2}-\alpha \log(e+|u|^2))|u|^2 |\nabla u|^2&\leq (e^{\frac{1}{2\alpha}}-e)_{+}\int_{M^c_{t}}|\nabla u|^2,	
	\end{align*}
	so
	\begin{align*}\frac{1}{2}\frac{d}{dt}\|\nabla u\|^{2}_{L^{2}}+\frac{1}{2}\|\Delta  u\|^{2}_{L^{2}}+\frac{\alpha}{2} \|\frac{|u|^2}{e+|u|^2} |\nabla|u|^2|^2\|_{L^{1}}\\+\frac{\alpha}{2} \|\log(e+|u|^2) |\nabla|u|^2|^2\|_{L^{1}}&\leq (e^{\frac{1}{2\alpha}}-e)_{+}\int_{M^c_{n}}|\nabla u|^2\\&\leq(e^{\frac{1}{2\alpha}}-e)_{+}\int_{\R^3}|\nabla u|^2.\end{align*}
	Integrate on $[0,t]$, we obtain:
	\begin{align*}
		\|\nabla u\|^{2}_{L^{2}}+\int_{0}^{t}\|\Delta u\|^{2}_{L^{2}}+\alpha \int_{0}^{t}\|\frac{|u^2|}{e+|u|^2} |\nabla|u|^2|^2\|_{L^{1}}\\+\int_{0}^{t}\alpha \|\log(e+|u|^2) |\nabla|u|^2|^2\|_{L^{1}}\|_{L^{1}}&\leq\|\nabla u^{0}\|^{2}_{L^{2}}+2(e^{\frac{1}{2\alpha}}-e)_{+}\int_{0}^{t}\|\nabla u\|^{2}_{L^2}
	\end{align*}
	By Lemma \ref{lem45}, we get
	\begin{align}\label{eq2}
		\nonumber\|\nabla u\|^{2}_{L^{2}}+\int_{0}^{t}\|\Delta u\|^{2}_{L^{2}}+\alpha \int_{0}^{t}\|\frac{|u^2|}{e+|u|^2} |\nabla|u|^2|^2\|_{L^{1}}\\ +\int_{0}^{t}\alpha \|\log(e+|u|^2) |\nabla|u|^2|^2\|_{L^{1}}\|_{L^{1}}\nonumber&\leq\|\nabla u^{0}\|^{2}_{L^{2}}e^{(2(e^{\frac{1}{2\alpha}}-e)_{+}t)}\\&\leq\|\nabla u^{0}\|^{2}_{L^{2}}e^{a_{\alpha}t},
	\end{align}
	where, $a_{\alpha}=(e^{\frac{1}{2\alpha}}-e)_{+}.$ \\
	Absolutely, these bounds come from the approximate solutions via the Friederich's regularization procedure. 
	The passage to the limit follows using classical argument by combining Ascoli's Theorem and the Cantor Diagonal Process \cite{HB}. And this solution in $L^{\infty}(\R^{+},H^1)\cap L^2(\R^{+},\dot{H^{2}})$ satisfies (\ref{eq1}) and (\ref{eq2}).\\
	$\bullet${\bf{Uniqueness :}}
	$$$$
	 Let $u,v$ two solutions of $(NSD_{\log})$ and $w=u-v$. We have:
	
	$$\partial_t u
	-\Delta u+ u.\nabla u  +\alpha \log(e+|u|^2)|u|^2u =\;\;-\nabla p_{1}\hbox{ in } \mathbb R^+\times \mathbb R^3~~~~~~~~~~(1)$$
	$$	\partial_t v
	-\Delta v+ v.\nabla v +\alpha \log(e+|v|^2)|v|^2v =\;\;-\nabla p_{2}\hbox{ in } \mathbb R^+\times \mathbb R^3~~~~~~~~~~~(2)$$
	
	Take $(1)-(2)$, we get:
	(1)-(2) given: $$\partial_t w -\Delta w+w.\nabla u+v\nabla w +\alpha(\log(e+|u|^2)|u|^2u-\log(e+|v|^2)|v|^2u)=-\nabla (p_1-p_2)$$
	Taking, the $L^2$ scalar product, we have :
	\begin{align*}
		\frac{1}{2}\frac{d}{dt}\|w\|^{2}_{L^2}+\|\nabla w\|^{2}_{L^2} +\alpha\langle(\log(e+|u|^2)|u|^2u-\log(e+|v|^2)|v|^2v),w\rangle_{L^{2}}&\leq|\langle w\nabla u, w\rangle|_{L^2}
	\end{align*}
	Using Lemma \ref{lem46}, we get:
	\begin{align*}
		\frac{1}{2} \frac{d}{dt} \|w\|^{2}_{L^{2}}+\|\nabla w\|^{2}_{L^{2}}&\leq \|w \|^2_{L^{4}}\|\nabla u\|_{L^{2}}\\&\leq \|w \|^2_{H^{\frac{3}{4}}}\|\nabla u\|_{L^{2}}\\
		&\leq \|w \|^{\frac{1}{2}}_{L^2} \|w\|^{\frac{3}{2}}_{\dot{H}^{1}}\|\nabla u\|_{L^{2}}\\&\leq \frac{3}{4}\|\nabla w\|^{2}_{L^2}+\frac{1}{4}\|w\|_{L^2}^2\|\nabla u\|^{4}_{L^2}\\
		&\leq \frac{1}{2}\|\nabla w\|^{2}_{L^2}+\frac{1}{4}\|w\|_{L^{2}}^2\|\nabla u\|^{4}_{L^2}.
	\end{align*}
	According to Gronwall Lemma and using (\ref{r1}), we have
	$$\|w(t)\|^{2}_{L^{2}}\leq \|w(0)\|^2_{L^{2}}e^{c\int_{0}^{t}\|\nabla u\|^{4}_{L^2}},$$	
	but $w(0)=0$, so $u=v$.\\
	$\bullet${\bf{Right continuity:}}
	$$$$
	$\ast$ Right continuity at $0$:
	\begin{align*}
		\|\nabla u\|^{2}_{L^{2}}+\int_{0}^{t}\|\Delta u\|^{2}_{L^{2}}+\alpha \int_{0}^{t}\|\frac{|u^2|}{e+|u|^2} |\nabla|u|^2|^2\|_{L^{1}}\\+\int_{0}^{t}\alpha \|\log(e+|u|^2) |\nabla|u|^2|^2\|_{L^{1}}\|_{L^{1}}&\leq\|\nabla u^{0}\|^{2}_{L^{2}}e^{c_{\alpha}t}.
	\end{align*}
	Let $t_{k}>0$ such that $t_{k}\underset{k\rightarrow\infty}\rightarrow 0 $ then
	\begin{equation} \label{e2}
		\underset{k\rightarrow\infty}{\limsup}\|\nabla u(t_{k})\|^{2}_{L^{2}}\leq \|\nabla u^{0}\|_{L^2}^{2}.\end{equation}
	From (\ref{e2}) we have the Right continuity at $0$.\\
	$\ast$ Right continuity at $t_{0}$:
	Let $t_{0}>0$. Consider the following problem:
	$$
	\begin{cases}
		\partial_t v
		-\Delta v+ v.\nabla v  +\alpha \log(e+|v|^2)|v|^2v =\;\;-\nabla q\hbox{ in } \mathbb R^+\times \mathbb R^3\\
		{\rm div}\, v = 0 \hbox{ in } \mathbb R^+\times \mathbb R^3\\
		v(0,x) =u(t_{0},x) \;\;\hbox{ in }\mathbb R^3.
	\end{cases}
	$$
	By uniqueness of solution $v(t)=u(t+t_{0})$ moreover $u$ is continuous on the right at $0$. Then $u$ is continuous on the right at $t_{0}$.\\
	$\ast$ The subset of $\R^+$
$$A=\{t\geq0/\,\|\nabla u\|_{L^{2}}\;{\rm is\;not\;continuous\;at}\;t\}$$
is at most countable.\\
Indeed: Let $0\leq t_{1}\leq t_{2}$ and set the function $f:t\mapsto e^{-a_{\alpha}t}\|\nabla u(t)\|^{2}_{L^2}$. 
Clearly, we have
$$A=\{t\geq0/\,f\;{\rm is\;not\;continuous\;at}\;t\}.$$
By the last proof, we have
	$$ \|\nabla u(t_{2})\|^{2}_{L^2}\leq\|\nabla u(t_{1})\|^{2}_{L^2} e^{a_{\alpha}(t_{2}-t_{1})},$$
	which implies
	$$\|\nabla u(t_{2})\|^{2}_{L^2}e^{-a_{\alpha}t_{2}}\leq \|\nabla u(t_{1})\|^{2}_{L^2} e^{-a_{\alpha}t_{1}}\Longleftrightarrow f(t_2)\leq f(t_1).$$
	Thus, $f$ is a decreasing function. According to (\ref{lem5}), $A$ is at most countable, which complete the proof.\\
	{\bf Acknowledgements.}
	It is pleasure to thank Jamel Benameur for insightful comments and assistance through this work.

\end{document}